\documentclass[10pt,draftcls,onecolumn]{IEEEtran}
\usepackage{epsfig,amssymb,latexsym,color,pifont,colordvi,multicol}
\usepackage[cmex10]{amsmath}
\usepackage{subfigure}

\definecolor{backgrey}{rgb}{0.86,0.86,0.86}
\definecolor{dblue}{rgb}{0,0.0,0.5}
\definecolor{dred}{rgb}{0.4,0.2,0}
\definecolor{dgreen}{rgb}{0.0,0.5,0}

\newcommand{\captionfonts}{\small}
\makeatletter  
\long\def\@makecaption#1#2{%
  \vskip\abovecaptionskip
  \sbox\@tempboxa{{\captionfonts #1: #2}}%
  \ifdim \wd\@tempboxa >\hsize
    {\captionfonts #1: #2\par}
  \else
    \hbox to\hsize{\hfil\box\@tempboxa\hfil}%
  \fi
  \vskip\belowcaptionskip}
\makeatother   
\newtheorem{theorem}{Theorem}

\newtheorem{assumption}[theorem]{Assumption}

\newtheorem{lemma}[theorem]{Lemma}

\newenvironment{proof}[1][Proof]{\textbf{#1.} }{\ \hspace*{\fill} \rule{0.5em}{0.5em}}

\title{Existence of invariant tori in three dimensional maps with degeneracy}

\author{Umesh Vaidya \footnote{Umesh Vaidya (ugvaidya@iastate.edu)} and Igor Mezi\'{c}
\footnote{Igor Mezi\'{c} (mezic@engr.ucsb.edu)} 
\thanks{U. Vaidya is with the Department of Electrical and
Computer Engineering,
Iowa State University,
Ames, IA}
\thanks{I. Mezi\'{c} is with the Department of Mechanical  Engineering,
University of California at Santa Barbara, Santa Barbara CA}
}

\begin{document}
\maketitle \thispagestyle{empty} \pagestyle{empty}
%



\begin{abstract}
We prove a KAM-type  result for the persistence of two-dimensional
invariant tori in  perturbations of integrable
action-angle-angle maps with degeneracy, satisfying the intersection property.
 Such degenerate action-angle-angle maps arise upon generic perturbation
of three-dimensional volume-preserving vector fields, which
are invariant under volume-preserving action of $S^1$
 when  there is no motion in the group action direction for the unperturbed map.
This situation is analogous to degeneracy in Hamiltonian systems. The
degenerate nature of the map  and the unequal number of action and
angle variables make the persistence proof non-standard.
The persistence of the invariant tori as predicted by
our result has implications for the existence of barriers to
transport  in three-dimensional incompressible fluid flows.
Simulation results indicating  existence of two-dimensional tori in
a perturbation of swirling Hill's spherical vortex flow are presented.
\end{abstract}




\section{Introduction}\label{intro}

The KAM (Kolmogorov-Arnold-Moser) \cite{kol, arnoldsmallden, arnoldsmallden_2,
moserkam} theorem is one of the most important results in the
stability theory of Hamiltonian systems. The theorem asserts that
most of the invariant $n$-tori of $n$ degrees of freedom integrable
Hamiltonian systems will persist under small Hamiltonian
perturbations. Arnold proved this theorem under both
non-degenerate and degenerate assumptions on the unperturbed
Hamiltonian \cite{arnoldsmallden, arnoldsmallden_2}. Moser proved a version of  the theorem for
the perturbation of two dimensional integrable twist map
\cite{SiegelandMoser:1971} (Chapter 3; Section 32). In both of these cases the system is
defined on an even dimensional manifold and has a symplectic
structure.

Extension of the KAM theorem to odd dimensional systems is a
challenging problem, that has many practical applications
\cite{mezicwigg,RLlavebook}. Volume-preserving flows and maps
which arise in the context of fluid dynamics and
magnetohydrodynamics are of odd dimensions. Because of that, these
maps and flows have a looser structure than symplectic maps and
flows. The KAM-type results have been developed for
volume-preserving flows \cite{broer1,broerbook,Delshams} and for
diffeomorphisms which either  preserve volume
\cite{Delshams,chengsun3, Xia, Zhu_Wen,Yong} or satisfy the
intersection property, a relaxed version of volume- preservation
\cite{chengsun1,Xiabook}. The result in this paper differs from the above mentioned references in that we prove the KAM-type result for the degenerate case of three-dimensional volume preserving maps (in fact, more generally for action-angle-angle maps with one degenerate angle and satisfying the intersection property). A KAM-type result for maps with unequal number of actions and angles and with degeneracy of the same type as that considered by us also appears in \cite{Zhu_Wen}. However there are some major differences between the KAM proof that appears in \cite{Zhu_Wen} and the main results of this paper. In particular,
in \cite{Zhu_Wen}, the KAM-type results are proved for the case where the size of the perturbations are assumed to be smaller than the size of the degenerate drift in the angles, whereas in this paper
we assume that both the degenerate drift and the perturbations are of same size. Furthermore the proof in the paper  \cite{Zhu_Wen} achieves their stated result only when an additional - unstated - assumption on the perturbation is used (see section \ref{first} below). Similarly \cite{Yong} prove KAM type result for the case where the unperturbed system consists of arbitrary number of action and angle variables. However the set-up does not consider the case of degenerate angle which is the case discussed in our paper.

%
%
%
%

The degenerate three dimensional volume preserving action-angle-angle map considered in this paper arises in the context of fluid flow problems. %
 The following example from \cite{mezicwigg},
shows how such action-angle-angle maps can arise in three-dimensional
incompressible volume-preserving flows, which are invariant under a
one-parameter symmetry group. Consider the following flow in
cylindrical coordinates.

\begin{eqnarray}
\dot r=r z,\;\;\;\;\;\;\;\; \dot z=1-2 r^2-z^2,\;\;\;\;\;\;
\;\;\dot \theta =\frac{2c}{r^2},\label{example}
\end{eqnarray}
where $c$ is an arbitrary constant. The system preserves the volume form $r dr \wedge dz \wedge d\theta$ \cite{broer_LNM,haller_mezic}.
In the fluid-mechanics context,
$\frac{c}{2}$ is the circulation. The flow (\ref{example}) is a
superposition of the well-known Hill's spherical vortex with a line
vortex on the $z$ axis, which induces the swirl velocity $\dot
\theta=\frac{2c}{r^2}$. The system of equation satisfies Euler's equation
of motion for the an inviscid incompressible fluid everywhere
except on the $z$ axis, where the swirl velocity becomes infinite.
After transforming the first two components into canonical
Hamiltonian form by letting $R=\frac{r^2}{2}$, the system (\ref{example})
becomes
\begin{eqnarray}
\dot R=2Rz,\;\;\;\;\;\;\;\; \dot z=1-4R-z^2,\;\;\;\;\;\;\;\; \dot
\theta=\frac{c}{R}. \label{eqn2}
\end{eqnarray}
This system preserves the volume form $dR \wedge dz \wedge d\theta$  \cite{broer_LNM,haller_mezic} and in the $R-z$ components takes the form

\begin{eqnarray}
\dot R=\frac{\partial H(R,z)}{\partial z},\;\;\;\;\;\;\;\;\;\;\;
\dot z=-\frac{\partial H(R,z)}{\partial R}\label{ham}
\end{eqnarray}

\noindent where $H(R,z)=Rz^2-R+2R^2$ is the Hamiltonian. By first introducing action-angle coordinate with respect to form $dR\wedge dz$, we transform $(R,z)$ to action-angle coordinate i.e., $(R,z)\to (I,\phi_1)$. To obtain action-angle-angle flow, we would need to perform addition transformation on the angle variable $\theta$ to get the second angle variable $\phi_2(\theta,I,\phi_1)$ (for the details of the derivation
refer to \cite{mezicwigg}). Hence we get,
\begin{eqnarray}
\dot I=0,\;\;\;\;\;\;\;\;\;\;\; \dot \phi_1 =
\omega_1(I)\;\;\;\;\;\;\;\;\;\;\; \dot \phi_2 =c \omega_2(I).
\end{eqnarray}
\noindent  For the case where $c$ is very large (i.e., $c>>1$ or $c\approx\frac{1}{\epsilon}$), we get the following degenerate action-angle-angle flow equations, after rescaling time $t=\frac{\tau}{c}$ and in the limiting case of $\epsilon=0$.
\begin{eqnarray}
\dot I=0,\;\;\;\;\;\;\;\;\;\;\; \dot \phi_1 =0\;\;\;\;\;\;\;\;\;\;\; \dot \phi_2 =\omega_2(I).\label{aaa}
\end{eqnarray}
\begin{figure}[h]
\begin{center}
\mbox{
\hspace{0in}
\subfigure{\scalebox{0.5}{\includegraphics{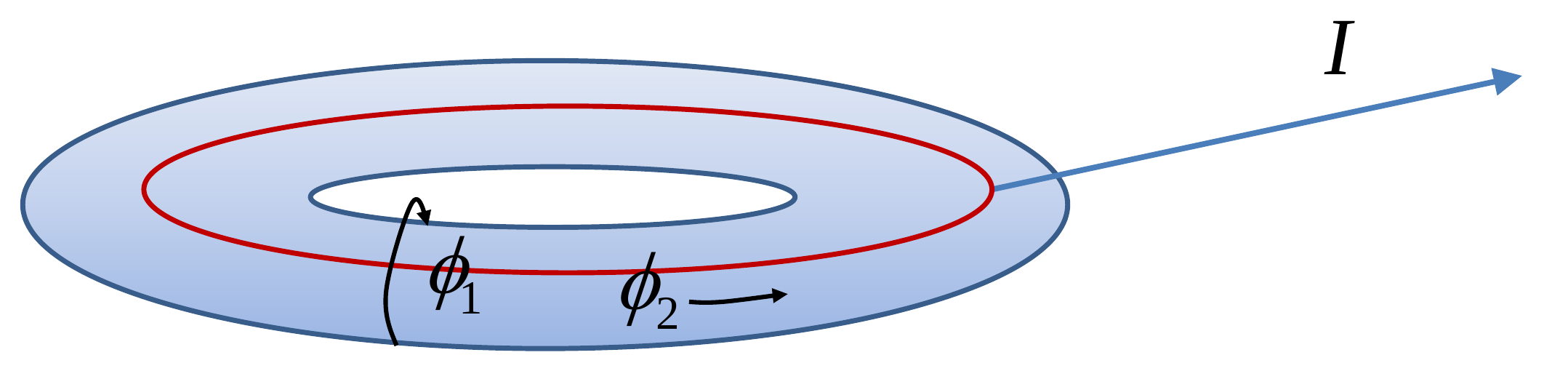}}} }
\caption{Geometry of action-angle-angle coordinates on cylinder torus with periodic orbit (red)}
\label{fig}
\end{center}
\end{figure}
\noindent The dynamics of (\ref{aaa}) evolves on a cylinder torus and consist of periodic orbits (refer to Fig. \ref{fig} for the schematic).
In this paper, we are interested in time periodic volume preserving perturbations of degenerate action-angle-angle flows as given in (\ref{aaa}) and the three dimensional maps that arise from it after taking appropriate Poincare section. We study the perturbation of the above discussed Hill spherical vortex flow for the case of large swirl in further detail in section \ref{Hill}.

%

Geometrical structures such as
invariant manifolds play an important role in understanding
the transport dynamics - specifically mixing and the lack thereof - in such maps.
From numerical studies and perturbation method
calculations, no invariant two-dimensional structure persists upon
perturbation from an integrable action-angle-angle map  with degenerate angle
\cite{piro}.  However the numerical studies carried out in \cite{piro} do not consider the class of perturbations for which the main result of this paper is proved and in fact corresponds to class of perturbations that is indicated in  \cite{mezicaaa} as ones with possibility of having invariant tori. Dynamics related to transport in phase space for
action-angle-angle maps with dynamically degenerate angle has been studied systematically in
\cite{mezicaaa} where it is shown that for a large class
of  such maps, upon perturbation, most of
the invariant surfaces are broken. The invariant surfaces break where resonance exists and at
these locations in phase space, periodic orbits of specific types
persist and dominate transport. This has been
named  Resonance-Induced Dispersion \cite{piro}.
The result in this
paper proves that, for a different class of perturbations, whose structure was also discussed in \cite{mezicaaa},
two-dimensional invariant tori indeed exists for the perturbed
action-angle-angle maps with degenerate angle satisfying intersection property - a condition that is implied by volume preservation. This proves the conjecture on such maps stated in \cite{mezicaaa}.

The KAM type of result for the action-angle-angle maps is
analogous to the degenerate Hamiltonian case treated by Arnold
\cite{arnoldsmallden_2}. In proving this degenerate case of KAM, we
are faced with two important problems. The first is due to unequal numbers of fast and slow
variables. Because of this a drift term is introduced at each step
of the coordinate transformations, we solve this problem by using
proof techniques similar to the one which appears in
\cite{chengsun1}. The second problem is due to the degenerate
nature of the one of the angles. We solve this problem by
introducing an intermediate finite sequence of coordinate
transformations. This finite sequence of coordinate
transformations is different from the intermediate coordinate
transformations which appear in Arnold's proof \cite{arnoldsmallden_2}
of degenerate KAM. The difference arises because of the difficulty with carrying out the Moser strategy of
solution of the sequence of equations by backward substitution which in this case leads to $O(1)$ terms after an iteration step. Thus our proof is different in nature from
the degenerate KAM proof which appears in \cite{arnoldsmallden_2}.

The paper is organized as follows. In section \ref{form}, we state the
main theorem for the persistence of invariant tori in
action-angle-angle maps with a degenerate angle. In section \ref{outl}, we give an outline of the
proof.  Simulation results for the Hill's spherical vortex example are presented in section  \ref{Hill} followed by conclusions in section \ref{conc}.

\section{Formulation of the theorem}
\label{form}

Consider the following mapping
\begin{equation}M=\left\{\begin{array}{ccl}
x_{1}&=&x+f(z)+\epsilon {\cal X}(x,y,z)\\
y_{1}&=&y+\epsilon
g_{0}(z)+\epsilon {\cal Y}(x,y,z)\;\;\;\;(x,y \;\;\;{\rm mod \; 2\pi})\\
z_{1}&=&z+\epsilon {\cal Z}(x,y,z)
\end{array}\right.\label{a1}
\end{equation}
where ${\cal X}$,${\cal Y}$, and ${\cal Z}$ are real analytic functions of period $2\pi$ in $(x,y)$ with $\epsilon$ being a small positive number. The $f$ and $g_0$ are analytic functions of $z\in [a,b]=G$. To simplify the analysis we assume that  $f(z)=z$ and $|g_0|\leq 1$. Since $\cal X,Y$ and,
$\cal Z$ are real analytic functions, they can be extended to a
complex domain:
\begin{eqnarray}
D\;\;: |{\mathbf Im}\;{x}|< r\leq 1,\;\;\;\;\;|{\mathbf Im}\;{y}|<
r \leq 1,\;\;\;\;\;z\in {\cal G},\label{complex_domain}
\end{eqnarray}
where ${\cal G}$ is the complex neighborhood of the interval
$[a,b]$. We now make following assumptions on the mapping (\ref{a1}).
\begin{assumption}\label{assumption_1} The functions $\cal Y$ and $Z$ are assumed to satisfy
\begin{center}
$\int_{0}^{2\pi}{\cal Y}dx=\int_{0}^{2\pi}{\cal Z}dx=0$.
\end{center}
The condition $\int_0 ^{2\pi}{\cal Y}dx=0$
can be relaxed by requiring that the integral $\int_0 ^{2\pi}{\cal
Y}dx$ be only a function of $z$ because any function of $z$ can
always be absorbed in $g_0(z)$.
\end{assumption}

\begin{assumption}\label{assumption_2}
Mapping (\ref{a1}) need not be
measure preserving but we assume that the map satisfies the
intersection property, i.e., any torus of the form:
\begin{eqnarray}
z=\gamma(x,y)\;\;\;{\rm
where}\;\;\;\gamma(x+2\pi,y)=\gamma(x,y)\;\;\&\;\;\;\gamma(x,y+2\pi)=\gamma(x,y)
\end{eqnarray}
intersect its image under the mapping.
\end{assumption}
\begin{assumption}\label{assumption_3} The function $g_0$ satisfies $g_0^{''}(z)\geq c_1>0$. $g_0^{''}>0$ is also referred to as second twist condition \cite{chengsun1}.
\end{assumption}
Now we state the main theorem for the persistence of
invariant tori in the action-angle-angle map with one degenerate angle.
\begin{theorem} \label{main_theorem} Consider the mapping (\ref{a1}) satisfying assumptions \ref{assumption_1}, \ref{assumption_2}, and \ref{assumption_3}.
 There exists a positive number $\epsilon_0$ depending upon domain $D$,
such that  on $D$ and for all $\epsilon \in (0,\epsilon_0)$, the
mapping (\ref{a1}) admits a family of invariant tori of the form:
\begin{eqnarray}
x=\xi+u(\xi,\zeta,\omega),\;\;\;\;\;
y=\zeta+v(\xi,\zeta,\omega),\;\;\;\;\; z=w(\xi,\zeta,\omega)\label
{a2},
\end{eqnarray}
where $u,v,w$ are real analytic functions of period $2\pi$ in the complex domain
 $|{\mathbf Im}\;x|< \frac{r}{2}$,  $|{\mathbf Im}\;y|< \frac{r}{2}$ with  $\omega\in S_{\omega}\subset G=[a,b]$, and $S_{\omega}$ is a Cantor set with positive Lebesgue measure. Moreover
the mapping can be parameterized so that the induced mapping on
the tori is given by
\begin{eqnarray}
\xi_{1}=\xi+\omega,\;\;\;\;\; \zeta_{1}=\zeta +\epsilon
g_{0}(\omega)+g^{*}(\omega,\epsilon)\label {a3},
\end{eqnarray}
where $g^{*}(\omega,\epsilon)$ is an analytic function that satisfies $g^*(\omega,0)=0$.
\end{theorem}

\section{ Outline of the proof}
\label{outl}
 The proof consists of applying coordinate transformations in three different steps.
The first step of averaging coordinate transformation is applied
to reduce the size of all the three perturbations to order
$\epsilon^{2}$. The second step consists of applying a finite
sequence of coordinate transformations to reduce the size of the
action perturbation to order $\epsilon^{3}$. In the third and
final step, we apply an infinite sequence of coordinate
transformations similar to the one applied in proving the
classical KAM theorem
\cite{arnoldsmallden,arnoldsmallden_2,SiegelandMoser:1971,chengsun1}, but with some
modifications.

\subsection{ First coordinate transformation}
\label{first}

With $f(z)$ replaced with $z$ in (\ref{a1}), we denote the original mapping $M$ (Eq. \ref{a1}) by $M_0$ and write it as
follows:
\begin{equation} M_{0}=\left\{\begin{array}{ccl}
x_{1}&=&x+z+\epsilon{\cal X}(x,y,z)\\ y_{1}&=&y+\epsilon
g_{0}(z)+\epsilon {\cal Y}(x,y,z)\;\;\;\;(x,y \;\;\;{\rm mod \;\;2\pi})\\
z_{1}&=&z+\epsilon {\cal Z}(x,y,z).
\end{array}\right.\label{firstcoordeq}
\end{equation}
This map is defined in the complex domain $D$ (Eq.
\ref{complex_domain}). Now we prove the main Lemma of the first
coordinate transformation. This Lemma is similar to the averaging
Lemma from \cite{arnoldsmallden_2}.

\begin{lemma}\label{averaging_lemma} Consider a coordinate transformation $I_\epsilon$, defined in domain $D$, of the form:
\[I_{\epsilon}=\{\bar{x}=x+\epsilon h_{1}(x,y,z),\;\;\;\;\; \bar{y}=y+\epsilon
h_{2}(x,y,z),\;\;\;\;\;\bar{z}=z+\epsilon h_{3}(x,y,z), \] where
$h_1, h_2$ and $h_3$ are real analytic functions and  periodic
with period $2\pi$ in $x$ and $y$. Using this coordinate
transformation, the mapping $M_0$ (Eq. \ref{firstcoordeq}), defined in the domain $D$, is transformed to the form $\bar M_{0}=I_{\epsilon}M_{0}I_{\epsilon}^{-1}$
\begin{equation}\bar M_{0}=\left\{\begin{array}{ccl}
\bar
x_{1}&=&\bar{x}+\bar{z}+\bar{X}(\bar{x},\bar{y},\bar{z})\\\nonumber
\bar y_{1}&=&\bar{y}+\epsilon
g_{0}(\bar z)+\bar{Y}(\bar{x},\bar{y},\bar{z})\;\;\;\;(x,y \;\;\;{\rm mod 2\pi})\\
\nonumber \bar z_{1}&=&\bar{z}+\bar{Z}(\bar{x},\bar{y},\bar{z}).
\end{array}\right.
\end{equation}
The mapping $\bar M_0$ is defined in a smaller domain:
\[\bar D: |{\mathbf Im}\; x|<r-\delta,\;\;\; \;|{\mathbf Im}\; y|<r-\delta,\;\;\;\; z\in {\cal G}^{'},\]
where $\delta$ is a small positive number. The domain ${\cal G}^{'}$ is
the complex neighborhood of $G^{'}$ and $G^{'}$ is obtained from
$G=[a,b]$ by removing finite number of resonance zones. In this
reduced domain $G^{'}$, $z$ satisfies following inequalities
\[|kz+2\pi n|\geq \bar K |k|^{-\bar \mu}\;\;\;(0<|k|\leq N),\]
where $\bar K$ is a positive constant, $N$ is a large integer, and
$\bar \mu\geq 3$.
 We have the  following estimates on the perturbations $\bar{X},\bar{Y},\bar{Z}$ in the domain $\bar D$

\begin{center}
$|\bar{X}|+|\bar{Y}|+|\bar{Z}|< \epsilon^{2}=d_{0}$.
\end{center}
\end{lemma}

\begin{proof}: The difference equation (\ref{firstcoordeq}) in the new coordinates can be written as follows:
\[\bar x_{1}=\bar x+\bar z
+\epsilon {\cal X}^{1}+O(\epsilon^{2}),\;\;\;
\bar y_{1}=\bar y+\epsilon g(\bar z)+\epsilon {\cal
Y}^{1}+O(\epsilon^{2}),\;\;\;
\bar z_{1}=\bar z+\epsilon {\cal Z}^{1}+O(\epsilon^{2}).\]
The size of the perturbations in the new coordinate will be of the
order $\epsilon^{2}$ if each of the following terms is of order
$\epsilon$.
\begin{eqnarray}
{\cal X}^{1}&:=&{\cal
X}(x,y,z)+h_{1}(x+z,y,z)-h_{1}(x,y,z)-h_{3}(x,y,z)\nonumber\\
{\cal Y}^{1}&:=&{\cal
Y}(x,y,z)+h_{2}(x+z,y,z)-h_{2}(x,y,z)\nonumber\\ {\cal
Z}^{1}&:=&{\cal Z}(x,y,z)+h_{3}(x+z,y,z)-h_{3}(x,y,z)\label{u3}
\end{eqnarray}
Perturbations $\cal X, Y,$ and $\cal Z$ can be expressed in the
Fourier series as
\[
 {\cal X}=\sum_{k=-\infty}^{\infty}{\cal
X}_{k}(y,z)e^{ikx},\;\; {\cal Y}=\sum_{k=-\infty}^{\infty}{\cal
Y}_{k}(y,z)e^{ikx},\;\; {\cal Z}=\sum_{k=-\infty}^{\infty}{\cal
Z}_{k}(y,z)e^{ikx}.\] Now we represent each of the $h_{i}$ by the
finite series $h_{i}=\sum_{|k|\leq N}h_{i_{k}}(y,z)e^{ikx}$, where
$h_{i_{k}}$ satisfies following equality for $|k|\leq N$:
\begin{eqnarray*}
h_{1_{k}}=\frac{{\cal X}_{k}-h_{3_{k}}}{(1-e^{ikz})},\;\;\;
h_{2_{k}}=\frac{{\cal Y}_k}{(1-e^{ikz})},\;\;\;
 h_{3_{k}}=\frac{{\cal Z}_{k}}{(1-e^{ikz})}.\label {a4}
\end{eqnarray*}
To satisfy the above equation for bounded $h_{ik}$, we require $z$
to satisfy the following inequalities:
\[|k z+2\pi n|\geq \bar K  |k|^{-\bar \mu}\;\;\; (0<|k|\leq N),\]
for some positive constant $\bar K$ and $\bar \mu\geq 3$.
 Since average
value of $\cal Z$ and $\cal Y$ with respect to $x$ is equal to
zero, $h_{3_{0}}$ and $h_{2_{0}}$ are free to take any value. We
make $h_{3_{0}}={\cal X}_{0}$ so as to satisfy first equality of
(\ref{u3}). With these choices of $h_{i}$ (\ref{u3}) reduces to
\[{\cal X}^{1}=\sum_{|k|>N}{\cal X}_{k}e^{ikx}, \;\;{\cal Y}^{1}= \sum_{|k|>N}{\cal Y}_{k}e^{ikx},\;\; {\cal
Z}^{1}=\sum_{|k|>N}{\cal Z}_{k}e^{ikx}.\]
Each of these terms will
be of order $\epsilon$, if $N$ is chosen sufficiently large to be
of order greater than
$\frac{1}{\delta}\ln{\frac{2}{(1-e^{-\delta})\epsilon}}$, (refer
to \cite{arnoldsmallden_2}, technical Lemmas on page 163).  The
$\delta$ is related to the new  domain $\bar D$ as follows:

\[\bar D\;\;: |{\mathbf Im}\;{x}|< r-\delta, \;\;\;|{\mathbf Im}\;{y}|< r-\delta,\; \;\;z\in {\cal G}^{'}.\]
This new complex domain ${\cal G}^{'}$ of $z$ is a complex
neighborhood of $G^{'}$, where $G^{'}$ is obtained from $G$ after
removing the finite number of resonance
intervals. The total measure
of the resonance
 intervals has an upper bound of $(b-a)2\bar K$, so
that the reduced domain is of order 1 for small value of $\bar K$
(refer to \cite{arnoldsmallden_2}, technical Lemmas on page 163).


In
this new domain following inequalities are satisfied
\[
|k z+2\pi n|\geq\bar K|k|^{-\bar \mu}\;\;\;(0<|k|\leq N).
\]
\end{proof}

After this first averaging coordinate transformation, we get
following estimates on the perturbations
\[|\bar X|+|\bar Y|+|\bar Z|< \epsilon^{2}=d_{0}.\]

The action variable $z$ now belongs to the domain which is a
function of $\epsilon$ i.e., $z\in G^{'}(\epsilon)$ and the
magnitude of the connected components of $G^{'}$ is going to zero as
$(\ln \frac{1}{\epsilon})^{-2}$.


At this point it seems that with some work is needed to deal with the
$(\ln \frac{1}{\epsilon})^{-2}$
shrinkage of connected components of  $G^{'}$, we should be able to utilize results of \cite{Zhu_Wen} to conclude existence of a Diophantine invariant tori in the perturbed mapping. However, careful examination of the proof of the result in \cite{Zhu_Wen} reveals that  the proof holds true only under the additional assumption that the perturbation $\bar Z=0$. The assumption of $\bar Z=0$ is clearly not satisfied in our case since $\cal Z$ is not assumed to be zero. In fact, the finite sequence of coordinate transformations discussed in the following section are precisely introduced to decrease the size of action perturbation $\bar Z$ relative to other perturbations, if not to make it zero.


\subsection{Second coordinate transformation}
\label{sec}
At this point we would like to continue with the standard infinite
sequence of coordinate transformations as in Moser
\cite{SiegelandMoser:1971} but we are faced with the following
problem. The aim is to reduce the size of all the three
perturbations $\bar X, \bar Y $ and $\bar Z$. Due to the
degenerate nature of the angle $y$, the small denominator problem
is exaggerated. The degenerate angle $y$ introduces a term of order
$\frac{1}{\epsilon}$ in the estimates, which gives $O(1)$
estimates for the size of the coordinate transformation. This
makes it impossible to continue with the infinite sequence of
coordinate transformations. This problem can be solved by
introducing an intermediate finite sequence of coordinate
transformations. The aim of the finite sequence of coordinate
transformations is to reduce the size of action perturbation $Z$
to order
$\epsilon^3$ so that $\frac{1}{\epsilon}$ order
term introduced by the degenerate angle can be compensated.

\noindent For notational convenience we remove the over-bar
notation from the coordinates and the perturbations $\bar X, \bar
Y$ and $\bar Z$ and parameterize the map by $\omega$. The new
map after the first coordinate transformation is denoted by
$M_0(\omega)$. At this stage it is not really necessary to
parameterize the mapping by $\omega$ however the importance of
this parameterization will become clear later in the infinite
sequence of coordinate transformations. We have,

\begin{equation}M_{0}(\omega)=\left\{\begin{array}{ccl}
x_{1}&=&x+\omega +z+X(x,y,z,\omega)\\\nonumber y_{1}&=&y+\epsilon
g_{0}(z,\omega)+Y(x,y,z,\omega)\;\;\;\;\\
\nonumber z_{1}&=&z+Z(x,y,z,\omega)
\end{array}\right.\label{e1}
\end{equation}
defined in the domain:  \[D_{0}(\omega):
|{\mathbf Im}\; x|< \hat{r}_{0}<1,\;\;\;\;\; |{\mathbf Im}\; y|<
\hat{r}_0<1,\;\;\;\;\; |z|< \hat{s}_0,\;\; \;\; \omega\in G^{'},
\]
where $\hat r_0, \hat s_0$ are positive numbers defined later and
$|X|+|Y|+|Z|<d_{0}$ in $D_{0}(\omega)$. The mapping $M_0(\omega)$
is parameterized such that $g_0(z,\omega)=g_0(z+\omega),
X(x,y,z,\omega)=X(x,y,z+\omega)$ and so on for $Y$ and $Z$. In the
second coordinate transformation we treat this map as an
action-angle map, where $x$ is the angle and $z$ is the action, and we
consider $y$ as a parameter. Note that $\omega\in G^{'}$, and
the magnitude of the connected components of $G^{'}$ is going to zero as
$(\ln \frac{1}{\epsilon})^{-2}$.
In order to account for the
shrinking size of the connected components of the domain $G^{'}$ with decreasing
$\epsilon$, we require $\omega$ to satisfy infinitely many
inequalities of the form:
\begin{equation}
|k\omega+2\pi n|\geq  \epsilon^{\gamma_1} \hat K|k|^{-\hat
\mu}\;\;\;(k,n=1,2,...)\label{sec_coor_inequ},
\end{equation}
where $\hat K$ is a positive constant, $\hat \mu \geq 3$, and
$\gamma_1$ is a suitably chosen constant satisfying
$0<\gamma_1<<1$. The introduction of $\epsilon^{\gamma_1}$ term in
(\ref{sec_coor_inequ}) ensures that while the size of the domain
$G^{'}$ goes to zero as
 $(\ln \frac{1}{\epsilon})^{-2}$
the
reduction in the size of perturbation can be obtained in the
neighborhood of the nonresonant value of action, the length of
which tends to zero as power of $\epsilon$. Now we show that
after applying finitely many coordinate transformations we can
reduce the size of action perturbation $Z$ to order
$d_{0}^{\frac{3}{2}}$. Let $T_{i}(\omega)$ denote these coordinate
transformations. Let
$M_{k}(\omega)=T_{k-1}^{-1}(\omega)M_{k-1}(\omega)T_{k-1}(\omega)$
be the mapping obtained after applying these coordinate
transformations and defined in the domain $D_{k}(\omega)$. We will
suppress the dependence on $\omega$ of the mapping $M$ and
coordinate transformation $T$ at some places for notational
convenience. We have following Lemma for the intermediate step of coordinate transformation.
\begin{lemma}\label{lemmasecondcoordinate} There exists a
coordinate transformation $T(\omega)$ of the form:

\begin{equation}\label{cor2}
T(\omega)= \{x=\varphi+\hat U(\varphi,\psi,\eta,\omega),\;\;\;\;
y=\psi,\;\;\;\; z=\eta+\hat W(\varphi,\psi,\eta,\omega)
\end{equation}
such that the mapping $M_{0}(\omega)$ (Eq. \ref{e1})
defined in the domain:
\[{\cal A}_0(\omega): |{\mathbf Im}\;x|<\hat r,\;\;\;\;|{\mathbf Im}\;y|<\hat r,\;\;\;\;|z|<\hat s,\]
 with $|X|+|Y|+|Z|<d$, takes the form $M(\omega)=T^{-1}(\omega)M_{0}(\omega)T(\omega)$
\begin{equation}
M(\omega)=\left\{\begin{array}{ccl} \varphi_{1}&=&\varphi+\omega+\eta + \hat{X}(\varphi,\psi,\eta,\omega)\\
\psi_{1}&=&\psi+\epsilon g_{0}(\eta,\omega)
+\hat{Y}(\varphi,\psi,\eta,\omega)\\\eta_{1}&=&\eta+
\hat{Z}(\varphi,\psi,\eta,\omega).
\end{array}\right. \label{act2}
\end{equation}
The mapping $M(\omega)$ is defined in the smaller domain ${\cal A}_{1}(\omega): |{\mathbf Im}
\;\varphi|< \hat{\rho}$, $|{\mathbf Im} \;\psi|< \hat{\rho}$,
$|\eta|<\hat{\sigma}$, with $0<\hat{\rho}<\hat{r}$,
$0<\hat{\sigma}<\hat{s}$. Assume that
\begin{equation}
\hat{r}<1,\;\;
0<3\hat{\sigma}<\hat{s}<\frac{\hat{r}-\hat{\rho}}{4},\;\;\;d<\frac{\hat{s}}{6},\;\;\;\hat \vartheta <\frac{\hat
\vartheta_2}{\hat s}<\frac{1}{7}\label{s5},
\end{equation}
\[{\rm where}\;\;\;
\hat{\vartheta}=
\frac{a_{1}}{\hat K^{2}}(\hat{r}-\hat{\rho})^{-2\hat
\mu-2}\frac{d}{\hat{s}}\epsilon^{-2\gamma_1}<\frac{1}{7},\;\;\;
\hat{\vartheta_2}=
\frac{a_{1}}{\hat K^{4}}(\hat{r}-\hat{\rho})^{-4\hat
\mu-3}d\epsilon^{-4\gamma_1},\;\;\; \] and $a_1$ is a positive
constant independent of the domain and depends only upon $\hat
\mu$. Using the above assumptions, we get following estimates for $\hat U,\hat
W,\hat{X},\hat{Y}$, and $\hat{Z}$

\begin{eqnarray*}
|\hat Y|&<& d,\;\;\;\;|\hat U|+|\hat
W|<\hat{\vartheta}\hat{s},\nonumber\\
|\hat{X}|+|\hat{Z}|&<& a_{4}\Big ((\hat{r}-\hat{\rho})^{-2\hat
\mu-3}\epsilon^{-2\gamma_1}d\hat{s}+(\hat{r}-\hat{\rho})^{-2\hat
\mu-3}\epsilon^{-2\gamma_1}\frac{d^{2}}{\hat{s}}\nonumber\\&+&(\hat{r}-\hat{\rho})^{-2\hat
\mu-3} \epsilon^{1-2\gamma_1}
d+(\frac{\hat{\sigma}}{\hat{s}})^{5}d \Big ),
\end{eqnarray*}
where $a_4$ is a positive constant independent of the domain.
\end{lemma}
The proof of this Lemma is similar to the Moser version of
the KAM proof for action-angle maps \cite{SiegelandMoser:1971}
with the difference being that the angle variable $y$ in this
proof is treated as a parameter. We refer the readers to \cite{SiegelandMoser:1971} (Chapter 3; Section 32) for the proof. We now use the result of this
Lemma to prove that at the end of the second coordinate transformation,
the size of action perturbation $Z$ is of order
$d_{0}^{\frac{3}{2}}$. To this end we apply the Lemma to
the mapping $M_{0}(\omega)$ defined in the domain:
\[
 D_{0}(\omega): |{\mathbf Im}\; x|< \hat{r}_0, |{\mathbf Im}\;
y|< \hat{r}_0, |z|< \hat{s}_0,
\]
where $D_0(\omega)$ correspond to the domain ${\cal A}_0(\omega)$
of the Lemma. By assumption, we have
\[|X|+|Y|+|Z|<d_{0} \;\;{\rm in}\;\; D_{0}(\omega).\]
Transforming the mapping $M_{0}(\omega)$ by the coordinate transformation
$ T_0(\omega)=T(\omega)$ provided by the Lemma, we obtain the
mapping $M_{1}(\omega)= T_{0}^{-1}(\omega) M_{0}(\omega) T_{0}(\omega)$
defined in the domain:
\begin{center}
$D_{1}(\omega): |{\mathbf Im} \;x|<\hat r_{1},\; |{\mathbf Im}
\;y|<\hat r_{1},\; |z|<\hat s_{1}$,
\end{center}
where  $D_1(\omega)$ correspond to the domain ${\cal A}_1(\omega)$
and $\hat r_{1}, \hat s_{1}$ corresponds to the parameter $\hat
\rho, \hat \sigma$ of the Lemma.
We define the following sequences
\begin{eqnarray*}
&\hat r_{n}=\frac{\hat r_{0}}{2}(1+\frac{1}{2^{n}}),\;\;\hat
s_{n}=d_{n}^{\frac{11}{50}}, \hat s_{0}=\hat{s}, \;\;d_{n+1}=\hat
r_0^{-\hat \chi}\hat
c_{7}^{n+1}d_{n}^{\frac{6}{5}}\epsilon^{-2\gamma_1}, \;\;\hat
\chi=2\hat \mu+3 \nonumber\\&\hat r_{0}=r-\delta,\;\;\; \hat
c_{7}>3,\;\;d_{0}<\hat c_{7}^{-20}\hat r_0^{7\hat
\chi}\epsilon^{14\gamma_1}.
\end{eqnarray*}
For the above sequences to be well defined we require that
$\gamma_1<\frac{1}{5}$. We need to check whether these sequences
satisfies the inequality (\ref{s5}). Towards this we have,

\begin{eqnarray*}
(\frac{\hat s_{n+1}}{\hat
s_{n}})^{\frac{50}{11}}=\frac{d_{n+1}}{d_{n}}=\hat
c_{7}^{n+1}d_{n}^{\frac{1}{5}}=e_{n}^{\frac{1}{5}}\hat
c_{7}^{-5}\hat r_0^{\hat\chi},
\end{eqnarray*}
where $e_n=\hat r_0^{-5\hat \chi} \hat c_7^{5(n+6)}\epsilon^{10
\gamma_1} d_n$ and $e_{n+1}=e_{n}^{\frac{6}{5}}$. Since $d_0<\hat
c_7^{-20}\hat r_0^{7\hat \chi}\epsilon^{14\gamma_1}$, we have $e_0<1$
and
\[\frac{s_{n+1}}{s_n}<\hat c_7^{\frac{-11}{10}}<\frac{1}{3},\;\;\;\;{\rm since} \;\; \hat c_7>3.\]
The inequality
$\hat s_{n}=d_{n}^{\frac{11}{50}}<\frac{\hat r_{n}-\hat
r_{n+1}}{4}, d_{n}<\frac{\hat s_{n}}{6},\;\; {\rm and}\;\;\hat
\vartheta<\frac{1}{7}
$ can be satisfied by taking $d_{0}$ sufficiently small and using
the fact that
\begin{eqnarray}
\hat{r}-\hat{\rho}=\hat r_{n}-\hat r_{n+1}=\hat r_{0}2^{-n-2},\;\;\;
\hat s_{n}=d_{n}^{\frac{11}{50}}\nonumber.
\end{eqnarray}
Using $d_{n+1}=\hat r_0^{-\hat \chi}\hat
c_{7}^{n+1}d_{n}^{\beta}\epsilon^{-2\gamma_1}$ with
$\beta=\frac{6}{5}$, we have
\[
d_{n+1}=\hat r_0^{-\hat \chi(\sum_{k=1}^n \beta^k)}\hat
c_{7}^{(n+1+n\beta+(n-1)\beta^2+...+\beta^n)}\epsilon^{-2\gamma_1(\sum_{k=1}^n \beta^k)}d_{0}^{\beta^{n+1}}.
\]
We want that after finitely many coordinate transformations $d_{n+1}<d_0^{\frac{3}{2}}$. Using the fact that
$\epsilon=d_0^{\frac{1}{2}}$ and $d_0<\hat
c_7^{-20}\hat r_0^{7\hat \chi}\epsilon^{14\gamma_1}$, it follows for $n=5$ that $d_6<d_0^{\frac{3}{2}}$ and we get,

\begin{eqnarray}
|\hat{X}_{6}|+|\hat{Z_{6}}|
<a_4\Big ( 2^{7(2\hat \mu+3)}\hat
c_{7}^{-6}d_{5}^{\frac{1}{50}}+ 2^{7(2\hat \mu+3)}\hat
c_{7}^{-6}d_5^{\frac{29}{50}}+ \hat r_0^{\hat \chi}2^{7(2\hat
\mu+3)}\hat
c_{7}^{-10}+ (\frac{d_{6}}{d_{5}})^{\frac{1}{10}}\Big )d_{6}.\nonumber
\end{eqnarray}
The coefficient multiplying $d_6$ can be made less than one by
choosing $\hat c_{7}$ sufficiently large, $\gamma_1$ sufficiently
small, and noticing that $\frac{d_{6}}{d_{5}}<1$ to give us
\begin{eqnarray*}
|\hat{X}_{6}|+|\hat{Z_{6}}|<d_{6}<d_{0}^{\frac{3}{2}}.
\end{eqnarray*}

\subsection{Infinite sequence of coordinate transformation}

At this stage of infinite sequence of coordinate transformations,
our aim is to decrease the size of all the three perturbations
simultaneously. In both the KAM proof for Moser twist map
\cite{SiegelandMoser:1971}, and the action-angle-angle maps \cite{chengsun1},
the size of all the perturbations decreases simultaneously with
the same estimates on the perturbations at each step of the
infinite sequence of coordinate transformations. In our proof, due
to the degenerate nature of the angle $y$, we require that the
size of action perturbation $Z$ is always $d^{\frac{1}{2}}$ order smaller
than the angle perturbations. This requires us to estimate the
size of action perturbation $Z$ separately from the size of the
angles perturbations $X$ and  $Y$.

Now we have a problem which is different from the Moser version of
the KAM proof for the twist maps, but similar to the one
faced in proving the KAM theorem for action-angle-angle maps. The
problem is due to unequal numbers of action and angle variables.
Due to this problem, it is not possible to predict which tori will
survive the perturbation and hence at this stage it becomes
necessary to parameterize the mapping by $\omega$.

We denote the mapping obtained after the second coordinate
transformation $M_{6}(\omega)$ by ${\cal M}_{0}(\omega)$. We are
using the same notation for the perturbation $X, Y$, and $Z$ as at the
beginning of the second coordinate transformation i.e., we define
$X:=\hat X_{6}, Y:=\hat Y_{6} , Z:=\hat Z_{6}$  and again the
parametrization on $X, Y, Z$ and $g_0$ are chosen such that
$X(x,y,z,\omega)=X(x,y,z+\omega)$ and
$g_0(z,\omega)=g_0(z+\omega)$ and so on for $Y$ and $Z$.
 So we have

\begin{equation}
{\cal M}_{0}(\omega)=\left\{\begin{array}{ccl}
x_{1}&=&x+\omega+z+X(x,y,z,\omega)\\ y_{1}&=&y+\epsilon
g_{0}(z,\omega)+Y(x,y,z,\omega)\\ z_{1}&=&z+Z(x,y,z,\omega),
\end{array}
\right.\label{inftycooreq}
\end{equation}
with $|X|<d_{0}^{\frac{3}{2}}<d_{0}$, $|Y|<d_{0}$, and
$|Z|<d_{0}^{\frac{3}{2}}$ defined in domain ${\cal D}_{0}(\omega):
|{\mathbf Im}\;x|< r_0\leq \hat r_6,\;\;|{\mathbf Im}\;y|< r_0\leq
\hat r_6,\;\; |z|< s_0\leq \hat s_6$.  To account for the
shrinking size of the connected components of domain $G^{'}(\epsilon)\ni \omega$ with
decreasing $\epsilon$, in this step of coordinate transformation
we require $(\omega,\epsilon g_{0}(\omega))$ to satisfy infinitely
many inequalities of the form:

\begin{eqnarray}
|k_{1}\omega+\epsilon k_{2}g_0(0,\omega)+2\pi
n|\geq\left\{\begin{array}{ccl}
\epsilon^{\gamma_2}K|k|^{-\mu}&{\rm if}& k_{1}\neq
0\\\epsilon^{1+\gamma_2} K|k|^{-\mu}&{\rm if}& k\neq
0\end{array}\right. \forall (k_{1},k_{2},n)\in {\mathbf
Z}^{3}\setminus \{0\},\label{d1}
\end{eqnarray}
where $|k|=|k_1|+|k_2|$, $K$ some positive constant, $\mu\geq 5$
and $\gamma_2$ is sufficiently small positive constant i.e.,
$0<\gamma_2<<1$. Now we use an infinite sequence of coordinate
transformations similar to the one used in
\cite{SiegelandMoser:1971} but with some modification. We have following induction Lemma for the third and final step of coordinate transformation.
\begin{lemma}\label{lemmainduction}
There exists a coordinate transformation ${\cal U}(\omega)$ of the
form:
\[{\cal
U}(\omega)=\{x=\varphi+U(\varphi,\psi,\eta,\omega),\;\;\;\;\;
y=\psi+V(\varphi,\psi,\eta,\omega),\;\;\;\;\;
z=\eta+W(\varphi,\psi,\eta,\omega)\;\;\;\;\;
\]
such that the mapping ${\cal M}_{0}(\omega)$ (Eq. \ref{inftycooreq}),
defined in the domain:
\[{\cal B}_0(\omega): |{\mathbf Im}\;x|<r,\;\;\;\;\; |{\mathbf Im}\;y|<r,\;\;\;\;\; |z|<s,\]
with $|X|+|Y|<d$ and $|Z|<d^{\frac{3}{2}}$ takes the form ${\cal
M}(\omega)={\cal U}^{-1}(\omega){\cal M}_{0}(\omega){\cal
U}(\omega)$. The mapping ${\cal M}(\omega)$
\begin{equation}
{\cal M}(\omega)=\left\{\begin{array}{lll}
\varphi_{1}=\varphi+\omega+\eta+\Phi(\varphi,\psi,\eta,\omega)\\\nonumber
\psi_{1}=\psi+\epsilon
g_{0}(\eta,\omega)+g_{1}(\eta,\omega)+\Psi(\varphi,\psi,\eta,\omega)\\\nonumber
\eta_{1}=\eta+H(\varphi,\psi,\eta,\omega)
\end{array}
\right.
\end{equation}
is defined in the following smaller domain:
\[{\cal B}_{1}(\omega): |{\mathbf Im}\;\varphi| < \rho\;\;|{\mathbf Im}\;\psi| < \rho\;\;|\eta| < \sigma,\]
with $0<\rho<r, 0<\sigma<s$. Now assume that
\begin{eqnarray}
0<r\leq \hat
r_{6},\;\;\;0<3\sigma<s<d^{\frac{1}{2}}(r-\rho),\;\; d<\frac{s}{2},\;\;\;\vartheta d^{\frac{1}{2}}<\frac{\vartheta_2 d^{\frac{1}{2}}}{s}<\frac{1}{8},\end{eqnarray}
\[{\rm  where}\;\;\vartheta=\frac{b_{1}}{K^{2}}(r-\rho)^{-2\mu-4}\epsilon^{-2\gamma_2}\frac{d^{\frac{1}{2}}}{s},\;\;\;\vartheta_2=\frac{b_{1}}{K^{4}}(r-\rho)^{-4\mu-4}\epsilon^{-2\gamma_2}d^{\frac{1}{2}}\] and $b_{1}$ is a positive constant
independent of the domain and depends only on $\mu$. Under the above assumptions, it follows that  ${\cal M}(\omega)$ is well defined in ${\cal B}_{1}(\omega)$
and there are following estimates:

\[|U|+|V|<
\frac{b_{1}}{K^{2}}(r-\rho)^{-2\mu-4}
\epsilon^{-2\gamma_2}d^{\frac{1}{2}},\;\;\;
|W|<\frac{b_{1}}{K}(r-\rho)^{-\mu-2}\epsilon^{-2\gamma_2}\frac{d^{\frac{3}{2}}}{\epsilon}
\]
\begin{eqnarray}\label{estinf}
|\Phi|+|\Psi|&<& b_6
\Big (\frac{(r-\rho)^{-2\mu-5}}{K^2}\epsilon^{-2\gamma_2}
d^{\frac{1}{2}}s
+\frac{(r-\rho)^{-2\mu-4}}{K^2}\epsilon^{-2\gamma_2}\frac{d^{\frac{1}{2}}}{s}|H|\nonumber\\&+&\frac{(r-\rho)^{-2\mu-4}}{K^2}\epsilon^{-2\gamma_2}\frac{d^2}{s}
+
\frac{(r-\rho)^{-4\mu-8}}{K^4}\epsilon^{-4\gamma_2}\frac{d^{\frac{5}{2}}}{s^2}\Big)
\end{eqnarray}

\begin{eqnarray*}
|H|&<&
b_{5}\Big (\frac{(r-\rho)^{-2\mu-5}}{K^2}\epsilon^{-2\gamma_2}d
s+\frac{(r-\rho)^{-4\mu-8}}{K^4}\epsilon^{-4\gamma_2}
\frac{d^{\frac{5}{2}}}{s}\nonumber\\&+&\frac{(r-\rho)^{-6\mu-13}}{K^6}\epsilon^{-6\gamma_2}\frac{d^{\frac{7}{2}}}{s^2}
+\frac{(r-\rho)^{-2\mu-4}}{K^2}\epsilon^{-2\gamma_2}\frac{d^{5/2}}{s}+(\frac{\sigma}{s})^{3}d^{\frac{3}{2}}\Big )
\end{eqnarray*}
where $g_{1}(0,\omega)=-\epsilon g_{0\eta}(0,\omega)\tilde
X(0,\omega)+\tilde Y(0,\omega)$, $\tilde X$ and $\tilde Y$ are
average value of $X$ and $Y$ respectively and $b_5, b_6$ are
positive constants independent of the domain. The functions $g_0(0,\omega)$ and $g_1(0,\omega)$ satisfy the
following new Diophantine conditions:

\begin{eqnarray*}
|k_{1}\omega+k_{2}(\epsilon
g_0(0,\omega)+g_1(0,\omega))+2\pi n|\geq\left\{\begin{array}{ccl}
\epsilon^{\gamma_2}\frac{K}{2}|k|^{-\mu}&{\rm if}& k_{1}\neq
0\\\epsilon^{1+\gamma_2} \frac{K}{2}|k|^{-\mu}&{\rm if}& k\neq
0\end{array}\right. \forall (k_{1},k_{2},n)\in {\mathbf
Z}^{3}\setminus \{0\}
\end{eqnarray*}

\end{lemma}
Proof of this Lemma follows along the similar lines for the
Moser version of the KAM proof for the action-angle map
\cite{SiegelandMoser:1971} (Chapter 3; Section 32). Due to unequal numbers of action and
angle variables we have a problem which is different from the KAM
proof for the twist map. The term $g_{1}(\eta,\omega)$ in the
mapping ${\cal M}(\omega)$ of the Lemma gives rise to the shift in
frequency of the degenerate angle. In general at the $(n+1)$ step
of the coordinate transformation there is a frequency drift from
$(\omega,\epsilon g_{0}(0,\omega)+\sum_{j=1}^{n}g_{j}(0,\omega))$ to
$(\omega,\epsilon g_{0}(0,\omega)+\sum_{j=1}^{n+1}g_{j}(0,\omega))$
similar to the case  in \cite{chengsun1}. In order to compensate
for this frequency drift we need to broaden the set of admissible
values of $\omega$. This can be achieved by allowing the constant
$K$ in inequalities (\ref{d1}) to decrease at each step of the
coordinate transformation. However decreasing the size of $K$ will
lead to increase in the size of estimates for the perturbations
and hence the scheme to make the mapping closer to the double
twist mapping might be a failure. We show that this is not always
the case and there exists a nonempty set $S(\omega)\subset G$ on
which corresponding $K$ decrease at most like power of $K_0$
$(K_n=\frac{K_0}{2^n})$ so that the size of the perturbations
decreases exponentially. To prove this we employ the strategy similar to that in \cite{chengsun1} with
the difference that while the strategy in \cite{chengsun1} is developed for action-angle-angle map with no degeneracy in angle, we extend it to the case of degenerate angle. More specifically there are following differences between our proof and proof technique developed in \cite{chengsun1}; 1) The averaging transformation is not needed in \cite{chengsun1}; 2) The intermediate sequence of transformations is not needed in \cite{chengsun1};
3) The ``Cantor set" calculations are substantially modified;
4) Whitney theory is used directly instead of doing it from scratch.
Before explaining this strategy, we prove the following Lemma similar to the one in \cite{chengsun1} except for the fact that the estimates in this Lemma are derived for the case of degenerate angle.

\begin{lemma}\label{cantorlemma}
Let $\epsilon>0$ be fixed and assume $\epsilon g_0(z)\in
C^{2}, g_0^{''}(z)\geq c_1>0$. Then the set $S(\omega)$,
where

\[S(\omega)=\{\omega \in G^{'}:|k\cdot \Omega +2\pi n|\geq\epsilon^{1+\gamma_2} K|k|^{-\mu},(\mu \geq 5),(k,n)\in {\mathbf
Z}^{3}\setminus\{(0,0,0)\}\} \]is a Cantor set with the Lebesgue
measure
$\mu_l(S(\omega))>\mu_l(G^{'})-c_{3}\left(\frac{\epsilon^{\gamma_2}K}{c_1}\right)^{\frac{1}{2}}$
where, $(k\cdot\Omega)=k_{1}\omega+\epsilon k_{2}g_0(0,\omega)$, $c_3$ is a positive constant, and
$\mu_l$ is the Lebesgue measure.
\end{lemma}

\begin{proof}
For a fixed $(k_{1},k_{2},n)\in {\mathbf Z}^{3}$, consider the
lines
\[
l_1:\;\; k_{1}\omega_1+\epsilon k_{2}\omega_2+2\pi n=0,\;\;\;
l_2:\;\; k_{1}\omega_1+\epsilon k_{2}\omega_2+2\pi
n-\epsilon^{1+\gamma_2} K|k|^{-\mu}=0\] \[ l_3:\;\;
k_{1}\omega_1+\epsilon k_{2}\omega_2+2\pi n+\epsilon^{1+\gamma_2}
K|k|^{-\mu}=0.
\]
The minimum distance between the lines $l_1, l_2$ or $l_1, l_3$ in
the $(\omega_1,\epsilon \omega_2)$ plane is $\epsilon^{1+\gamma_2}
K|k|^{-\mu}(k_1^{2}+k_2^{2})^{-\frac{1}{2}}$. Consider the points
$(\omega_1,\epsilon \omega_2)$ in the set
\[
\Omega_1=\{(\omega_{1},\epsilon \omega_{2})\in {\mathbf R}^{2}:d((\omega_{1},\epsilon \omega_{2}),l_1)\geq \delta\},\\
\delta=\epsilon^{1+\gamma_2}
K|k|^{-\mu}(k_{1}^{2}+k_{2}^2)^{-\frac{1}{2}}.
\]
The points in $\Omega_1$ set will satisfy the inequality
(\ref{d1}) for a fixed $(k_{1},k_{2},n,K)$, where $d(\#,l)$ means
the distance to the line $l$. Let $\Omega_2={\mathbf
R}^{2}\setminus \Omega_1$, $\Gamma=$ graph $\epsilon g_0$, and
$T_{1}$ the projection over first component. The length of
$T_1|(\Gamma \cap {\Omega_2})$ is less than
$4(\frac{\delta}{\epsilon c_1})^{\frac{1}{2}}$. For a fixed $k\in
{\mathbf Z}^{2}$, if $(\omega_{1},\epsilon \omega_{2})$ is
restricted in the domain $R_{\omega}$:
\[
0\leq \min (\omega_{1},\epsilon \omega_{2})<\max
(\omega_{1},\epsilon \omega_{2})\leq a_{8},
\]
then the set $\Omega_2 \cup R_{\omega}$ is non empty only if
$|n|<\frac{a_{8}|k|}{2\pi}$. So the Lebesgue measure of the set
$S(\omega)$ is

\[\mu_l(S(\omega))>\mu_l(G^{'})-\frac{4a_{8}}{2\pi}\left(\frac{\epsilon^{\gamma_2}K}{c_1}\right)^{\frac{1}{2}}\sum_{k\in {\mathbf
Z}^{2}}|k|^{\frac{-\mu}{2}+1}(k_1^{2}+k_2^{2})^{-\frac{1}{4}},
\]
which is positive if $\mu\geq 5$ and $K$ is sufficiently small.
Now
\[
\frac{4a_{8}}{2\pi}\sum_{k\in
\mathbf{Z}^{2}}|k|^{\frac{-\mu}{2}+1}(k_1^2+k_2^2)^{-\frac{1}{4}}<\frac{4a_{8}}{2\pi}\sum_{k\in
\mathbf{Z}^{2}}|k|^{\frac{-\mu}{2}+1-\frac{1}{4}}=:c_3
\]
and the sum converge because $\mu\geq 5$ and $
\mu_l(S(\omega))>\mu_l(G^{'})-c_3(\frac{\epsilon^{\gamma_2}K}{c_1})^{\frac{1}{2}}$
\end{proof}

We now introduce a sequence of coordinate transformation ${\cal
U}_{n}(\omega)$ on a nonempty set $\tilde S_{n}\subset G'$:
\[{\cal M}_{n+1}(\omega)={\cal U}^{-1}_{n}(\omega){\cal M}_{n}(\omega){\cal U}_{n}(\omega),\;\;
{\rm with}\;\;\tilde S_{n}(\omega) \subset \tilde S_{n-1}(\omega).\]
By Lemma \ref{cantorlemma}, there exists a Cantor set
$S_{0}(\omega)=\tilde S_{0}(\omega)$ given by
\[S_{0}(\omega)=\{\omega\in G^{'}|(\omega,\epsilon g_{0}(0,\omega)) {\rm \;satisfies\; (\ref{d1})\; with\;} K_{0} {\rm
\;replacing\;} K\}.\] The Lebesgue measure of the set $S_{0}(\omega)$
is
$\mu_{l}(S_{0}(\omega))>\mu_l(G^{'})-c_{3}\left(\frac{\epsilon^{\gamma_2}K_{0}}{c_{1}}\right)^{\frac{1}{2}},
$
 where $\mu_l$ is the Lebesgue measure. Let $\tilde
S_{1}(\omega)=\bigcap_{j=0}^{1}S_{j}(\omega)$, where

\[S_{1}(\omega)=\{\omega\in G^{'}| (\omega,\epsilon g_{0}(0,\omega)+g_{1}(0,\omega)){\rm \; satisfies} \;(\ref{d1})
\;{\rm with }\; \frac{K_{0}}{2}\;{\rm  replacing }\; K \}.\]
In order to derive the Lebesgue measure of the set $S_1(\omega)$
we need $g_1(0,\omega)$ to be defined on the entire domain $G$.
However, $g_1$ is only defined on the set $\tilde S_0(\omega)$.
This problem can be solved using the Whitney extension theorem
\cite{whitney}. By using the Whitney extension theorem,
we can extend the perturbations $X, Y, Z$ and subsequent
perturbations $X_j, Y_j, Z_j$ coming from infinite sequence of
coordinate transformation to the entire domain $G$ w.r.t. variable $\omega$.
The proof for the extension follows along the lines of proof outlined in  \cite{chengsun3,stein}.
Since $X, Y$, and $Z$
are extended to the domain $G$, the function
$g_{1}$ is well defined for all values of $\omega \in G$ because
\[g_{1}(0,\omega)=-\epsilon g_{0\eta}(0,\omega)\tilde X(0,\omega)+\tilde Y(0,\omega),\]
where $\tilde X$ and $\tilde Y$ are the average values of $X$ and
$Y$ respectively.
We will use the same notation for the functions and its extension to the domain $G$ w.r.t. variable $\omega$ with the following estimate
\[\parallel X\parallel_{2,1,G}+\parallel Y \parallel_{2,1,G}< c_w d \;\;\; i=0,1,2,\]
where $c_{w}$ is the Whitney constant and is independent of the
domain. The notation $\parallel \cdot \parallel_{2,1,G}$ is used as a measure for the norm of the function and the second derivative of the function w.r.t. variable $\omega$ in the domain $G$ (For more details on the norm refer to \cite{stein}).  Differentiating $g_{1}(0,\omega)$ twice w.r.t.
$\omega$ we get,
\[\left |\frac{d^{2}}{d\omega^{2}}g_{1}(0,\omega)\right |< 5\sup (\epsilon g_{0\eta},\epsilon g_{0\eta \omega},\epsilon g_{0\eta \omega^2},1)c_{w}d < \epsilon \beta_{1}.\]
Hence
\[\frac{d^{2}}{d\omega^{2}}(\epsilon g_{0}+g_{1})>\epsilon (c_{1}-\beta_{1}),\;\;\;{\rm and }\;\;\;
\mu_l(S_{1}(\omega))>\mu_l(G^{'})-c_{3}\left(\frac{K_{0}\epsilon^{\gamma_2}}{2(c_{1}-\beta_{1})}\right)^{\frac{1}{2}}.\]
The measure of the set $S_1(\omega)$ is obtained from Lemma
\ref{cantorlemma} by applying the results  to $k\cdot
\Omega=k_1\omega+k_2(\epsilon g_0+g_1)$. Now
$\tilde S_{1}(\omega)=S_{0}(\omega)\bigcap S_{1}(\omega)$ and
\[\mu_l(\tilde S_{1}(\omega))>\mu_l(G^{'})-c_{3}\left(\frac{K_{0}\epsilon^{\gamma_2}}{c_{1}}\right)^{\frac{1}{2}}
\left(1+\left(\frac{c_{1}}{2(c_{1}-\beta_{1})}\right )^{\frac{1}{2}}\right).\]
For $K_{0}$ and $\beta_{1}$ sufficiently small $\mu_l(\tilde
S_{1}(\omega))$ is positive. We obtain the following
expression for $g_{j+1}$ by induction on $g_1$ and its derivation
is similar to that of $g_1$
 \[g_{j+1}(0,\omega)=-\left(\epsilon g_{0\eta}(0,\omega)+\sum_{l=1}^{j}g_{l\eta}(0,\omega)\right)\tilde X_{j}(0,\omega)+\tilde Y_{j}(0,\omega),\]
where $X_{j}$ and $Y_{j}$ are extended to interval $G$ by using Whitney's extension
 with the following  estimates
\[\parallel X_{j}\parallel_{2,1,G}+\parallel Y_{j}\parallel_{2,1,G} < c_w d_j .\]
Assume that there exists a positive constant $c_{5}$ such that
\begin{eqnarray}
\sup \left |\frac{d^{v}}{d\omega^{v}}(\epsilon
g_{0\eta}(0,\omega)+\sum_{l=1}^{j}g_{l\eta}(0,\omega))\right |<c_{5}\;\;\;\;\;
{ v}=0,1,2.\label{lll}
\end{eqnarray}
The existence of such a positive constant $c_5$ can be proved as follows:
\[|\epsilon g_{0\eta}|\leq 1\;\;\;{\rm for}\;\; \eta\in G,\;\;{\rm and}\;\;|\epsilon g_{0\eta\omega^2}|\leq 1 \]
\[|g_1|=|\epsilon g_{0\eta}\tilde X_0|+|\tilde Y_0|\leq 2d_0,\;\;\; |g_{1\eta}|\leq \frac{2 d_0}{s_0}
\;\;\;{\rm for}\;\; |\eta|<s_0, \;\;{\rm and}\;\;|g_{1\eta
\omega^2}|\leq c_w\frac{2 d_0}{s_0}\]
\[|g_{j+1}|\leq 2 d_j \Pi_{i=1}^{j-1}(1+\frac{d_i}{s_i}),\;\;\;|g_{j+1,\eta}|\leq
\frac{2 d_j}{s_j} \Pi_{i=1}^{j-1}(1+\frac{d_i}{s_i})\;\;{\rm
for}\;\;|\eta|<s_j\;\;{\rm and}\;\;\] \[ |g_{j+1,\eta
\omega^2}|\leq \frac{ c_w 2 d_j}{s_j}
\Pi_{i=1}^{j-1}(1+\frac{d_i}{s_i}),\] so by choosing $d_0$
sufficiently small it is possible to find the constant $c_5$ such that (\ref{lll}) is true. Now
setting $c_{6}=5\max(c_{5},1)$ we get
\[\left |\frac{d^{2}}{d\omega^{2}}g_{j+1}(0,\omega)\right |<c_{6}c_wd_{j}=\epsilon \beta_{j+1},\;\;\;\;\sum_{l=1}^{j+1}\beta_l=\frac{c_6
c_w}{\epsilon}\sum_{l=1}^{j+1}d_{l-1}\] and this can be made less
than $\frac{c_1}{2}$, if we choose $d_0$ sufficiently small. So we have
$\sum_{l=1}^{j+1}\beta_l<\frac{c_{1}}{2}$ and then following
inequality holds
\[\frac{d^{2}}{d\omega^{2}}(\epsilon g_{0\eta}(0,\omega)+\sum_{l=1}^{j+1}g_{l\eta}(0,\omega))>\epsilon c_{1}-\epsilon \sum_{l=1}
^{j+1}\beta_{l}>\epsilon \frac{c_{1}}{2}.\]
By defining $\tilde S_{j+1}(\omega)=\bigcap_{l=0}^{j+1}S_{l}(\omega)$,
where
\begin{eqnarray*} S_{j+1}(\omega)=\{\omega\in
G^{'}|(\omega,\epsilon
g_{0}(0,\omega)+\sum_{l=1}^{j+1}g_{l}(0,\omega))\;\; {\rm
satisfies \;(\ref{d1})\; with}\;\;
K_{j+1}=K_{0}/2^{j+1} \;{\rm in\;place\;of}\; K \},
\end{eqnarray*}
we obtain,
\begin{eqnarray*}
\mu_l(S_{j+1}(\omega))>\mu_l(G^{'})-c_{3}\left(\frac{K_{0}\epsilon^{\gamma_2}}{(c_{1}-\sum_{l=1}^{j+1}\beta_{l})2^{j+1}}\right)^{\frac{1}{2}}\end{eqnarray*}
\begin{eqnarray*}
\mu_l(\tilde
S_{j+1}(\omega))&>&\mu_l(G^{'})-c_{3}\left(\frac{K_{0}\epsilon^{\gamma_2}}{c_{1}}\right)^{\frac{1}{2}}\sum_{l=1}^{j+1}2^{-l/2}
\left(\frac{c_{1}}{c_{1}-\sum_{n=1}^{l}\beta_{n}}\right)\\
&>&\mu_{l}(G^{'})-c_{3}\left(\frac{K_{0}\epsilon^{\gamma_2}}{c_{1}}\right)^{\frac{1}{2}}\sum_{l=0}^{j+1}2^{-l/2+1}.
\end{eqnarray*}
The measure of set $\tilde S_{j+1}$ (i.e., $\mu_l(\tilde S_{j+1}(\omega))$) is positive if $K_{0}$ is
sufficiently small.  The total drift in the degenerate angle at the $j^{th}$ step of iteration is given by $\epsilon g_0(0,\omega)+\sum_{k=1}^j g_{j}(0,\omega)$ and in the limit as $j\to \infty$ we get $\epsilon g_0(0,\omega)+g^{*}(\omega,\epsilon)$, where $g^{*}(\omega,\epsilon):=\sum_{j=1}^\infty g_j(0,\omega)$.

Now we define a sequence similar to the one in
\cite{SiegelandMoser:1971}. Let
$r_{n},r_{n+1},s_{n},s_{n+1},d_{n},\frac{K_{0}}{2^{n}}$
correspond to parameter $r,\rho,s,\sigma,d_{0},K$ respectively.
Setting
\[
r_{n}=\frac{r_0}{2}(1+\frac{1}{2^{n}}),\;\;\;K_{n}=\frac{K}{2^{n}},\;\;\;
 s_{n}=d_{n}^{\frac{11}{16}},\;\;\; s_{0}=\hat s_{6},\;\;\;
d_{n+1}=r_{0}^{-\chi}c_{7}^{n+1}\epsilon^{-6\gamma_2}d_{n}^{\frac{9}{8}}
\]
 where $c_{7}>2$, $\chi=2\mu+5$ are suitable constants,
$r_{n}$ converges to $\frac{r_{0}}{2}$, and $d_{n}$ converges to zero
provided $d_{0}$ is chosen sufficiently small. The sequence
$e_{n}=r_{0}^{-8\chi}c_{7}^{8(n+9)}\epsilon^{-48\gamma_2}d_{n}$
satisfies $e_{n+1}=e_n^{\frac{9}{8}}$ and hence
 converges to zero if we take
$0< d_{0}^{1-24 \gamma_2}<r_{0}^{8\chi}c_{7}^{-72}$. The
inequality $3\sigma<s$ follows from
\begin{center}
$(\frac{s_{n+1}}{s_{n}})^{\frac{16}{11}}=\frac{d_{n+1}}{d_{n}}=r_{0}^{-\chi}c_{7}^{n+1}d_{n}^{\frac{1}{8}}=e_{n}^{\frac{1}{8}}
c_{7}^{-8}<\frac{1}{c_{7}^{8}}$,
$\;\;\frac{s_{n+1}}{s_{n}}<\frac{1}{c_{7}^{5.5}}<\frac{1}{3}$, and
$r_{n}-r_{n+1}=r_{0}2^{-n-2}$.
\end{center}
Now we will use induction Lemma to show that
$|H_{n+1}|<d_{n+1}^{\frac{3}{2}}$ and
$|\Phi_{n+1}|+|\Psi_{n+1}|<d_{n+1}$. By induction on second
inequality of (\ref{estinf}) and the fact that
$s_n=d_n^{\frac{11}{16}}$ we have


\begin{eqnarray}
|H_{j+1}|&<&
b_{5}\Big (\frac{r_{0}^{\chi/2}2^{\chi(j+2)+2j}}{K_{0}^{2}}c_7^{-\frac{3}{2}(j+1)}\epsilon^{7\gamma_2}
+\frac{r_{0}^{2-\frac{\chi}{2}}2^{2(\chi-1)(j+2)+4j}}{K_{0}^{4}}
c_7^{-\frac{3}{2}(j+1)}\epsilon^{5\gamma_2}d_{j}^{\frac{1}{8}}\nonumber\\
&+&\frac{r_{0}^{2-
\frac{3\chi}{2}}2^{-(j+2)(2-3\chi)+6j}}{K_{0}^{6}}c_7^{-\frac{3}{2}(j+1)}\epsilon^{3\gamma_2}d_{j}^{\frac{7}{16}}
+\frac{r_0^{1+\frac{\chi}{2}}2^{-(j+2)(1-\chi)+2j}}{K_0^2}c_7^{-\frac{3}{2}(j+1)}\epsilon^{7\gamma_2}d_j^{\frac{2}{16}}\nonumber\\&+&
(\frac{d_{j+1}}{d_j})^{\frac{9}{16}}\Big )d_{j+1}^{\frac{3}{2}}.
\end{eqnarray}
Since $d_{j}$ is bounded, the coefficient of $d_{j+1}$ can be made
less than 1 by choosing $c_{7}$ sufficiently large and hence we
have
\begin{equation}
|H_{j+1}|< d_{j+1}^{\frac{3}{2}}\label{esth}.
\end{equation}
Using the fact that $s=d^{\frac{11}{16}}$,
$|H|<d^{\frac{3}{2}}$ and  by induction on the first inequality of (\ref{estinf}) and using
equation (\ref{esth}), we have


\begin{eqnarray}
|\Phi_{j+1}|+|\Psi_{j+1}|&<&b_{4}\Big (\frac{2^{\chi(j+2)+2j}}{K_{0}^{2}}c_{7}^{-(j+1)}\epsilon^{4\gamma_2}d_{j}^{\frac{1}{16}}
+\frac{2r_{0}2^{-(1-\chi)(j+2)+2j}}{K_{0}^{2}}c_{7}^{-(j+1)}\epsilon^{4\gamma_2}d_{j}^{\frac{3}{16}}\nonumber\\
&+&\frac{r_{0}^{(-\chi+2)}2^{-(j+2)(2-2\chi)+4j}}{K_{0}^{4}}c_{7}^{-(j+1)}\epsilon^{2\gamma_2}\Big )d_{j+1}.
\end{eqnarray}
The coefficient of $d_{j+1}$ can be made less than one by choosing
$c_{7}$ sufficiently large and hence

\[|\Phi_{j+1}|+|\Psi_{j+1}|<d_{j+1}.\]
Thus there exists a positive constant $d^{*}=d^{*}(r, c_1, K_0,
\hat K_0, \bar K, \mu, \hat \mu, \bar \mu,\gamma_1,\gamma_2)$ such
that the theorem is true for $d_0\in(0,d^{*})$  with $g^{*}(\omega,\epsilon):=\sum_{j=1}^{\infty} g_j(0,\omega)$.

\section{Application to Hill's spherical vortex flow }\label{Hill}
A particularly important application of the theorem proven above is in the case of a three-dimensional, time-periodic, volume-preserving fluid flows \cite{mezicaaa,Solomonandmezic:2003}. A steady integrable example of a three-dimensional vortex structure was developed in \cite{mezicwigg} as
an extension (called swirling Hill's vortex) of the well-known Hill's spherical vortex flow (see Eq. \ref{example}). The swirling Hill's vortex, besides radial and axial velocity in three-dimensional polar coordinates, contains a strong swirl induced by a line vortex situated
at the $z$ axis.
Here we consider the volume-preserving time-dependent perturbation of the swirling Hill vortex (\ref{example}) with strong swirl.  In cylindrical coordinates
the equations of motion of fluid particles  are given as follows:
\begin{eqnarray}
\dot r &=& r z+ \sqrt{2r} \sin \theta \sin  \Omega (c) t\nonumber\\
\dot z &=& 1-2r^2- z^2 -z \sqrt{\frac{1}{2 r}}\sin \theta \sin  \Omega(c) t\nonumber\\
\dot \theta &=&\frac{2c}{r^2 } + \sqrt{2r}\cos \theta \sin  \Omega(c) t,\label{hills}
\end{eqnarray}
 where $\frac{\Omega(c)}{c}=: \omega$ is assumed to be of  $O(1)$ size. Under the assumption that the swirl $c>>1$ or $\frac{1}{c}\approx \epsilon$ and after rescaling the time $t=\frac{\tau}{c}$, we get the following time periodic perturbed flow equations in the transformed action-angle-angle  coordinates:
%
\begin{eqnarray}
\dot I&=& \epsilon F_I(I,\phi_1,\phi_2,\tau)\nonumber\\
\dot \phi_1&=&\epsilon \omega_1(I)+\epsilon F_{\phi_1}(I,\phi_1,\phi_2,\tau)\nonumber\\
\dot \phi_2&=&\omega_2(I)+\epsilon F_{\phi_2}(I,\phi_1,\phi_2,\tau)\label{sim_example_pert},
\end{eqnarray}
where the action-angle variables $(I,\phi_1)$ are obtained from $(r,z)$ and the second angle variable $\phi_2$ is obtained using the following transformation \cite{mezicwigg}
\begin{eqnarray}
\phi_2=\theta+\frac{\phi_1}{2 \pi}\int_0^{2\pi} \frac{2}{r^2(I,\phi_1)\omega_1(I)}d\phi_1-\int \frac{2}{r^2(I,\phi_1)\omega_1(I)}d\phi_1 \label{second_angle}.
\end{eqnarray}
We are interested in showing that the Poincare map constructed from the system (\ref{sim_example_pert}) satisfies the Assumption \ref{assumption_1} of the main theorem. Towards this goal, we write $\theta$ as $\theta=\phi_2-\varphi(I,\phi_1)$, where $\varphi$ is defined using (\ref{second_angle}) as follows:
\[
\varphi(I,\phi_1):=\frac{\phi_1}{2 \pi}\int_0^{2\pi} \frac{2}{r^2(I,\phi_1)\omega_1(I)}d\phi_1-\int \frac{2}{r^2(I,\phi_1)\omega_1(I)}d\phi_1.\]
The action-angle perturbations terms appearing in (\ref{sim_example_pert}) can be written as:
\begin{eqnarray}
F_{\phi_1[I]}= \sin (\phi_2-\varphi) \sin \omega \tau \left(\frac{\partial \phi_1 [I]}{\partial r} \sqrt{2 r(I,\phi_1)}-\frac{\partial \phi_1[I]}{\partial z}z(I,\phi_1)\sqrt{\frac{1}{2r (I,\phi_1)}}\right).\label{step}
\end{eqnarray}
Defining $G_{\phi_1[I]}:=\left(\frac{\partial \phi_1 [I]}{\partial r} \sqrt{2 r(I,\phi_1)}-\frac{\partial \phi_1[I]}{\partial z}z(I,\phi_1)\sqrt{\frac{1}{2r (I,\phi_1)}}\right)$, we write (\ref{step}) as
\begin{eqnarray}
F_{\phi_1[I]}= \left(\sin \phi_2\cos \varphi-\cos \phi_2\sin \varphi \right) \sin \omega \tau G_{\phi_1[I]}(I,\phi_1)\label{F}.
\end{eqnarray}
The vector field (\ref{sim_example_pert}) is time periodic with time period $T=\frac{2\pi}{\omega}$ and hence we can construct the Poincare map. Using the regular perturbation theory, the solutions of (\ref{sim_example_pert}) are $O(\epsilon)$ close to the unperturbed solutions on the time scale of $O(1)$, and hence can be written as
\begin{eqnarray*}
I^{\epsilon}(t)&=&I^0+\epsilon I^1(t)+O(\epsilon^2)\\
\phi_1^{\epsilon}(t)&=&\phi_1^0+\epsilon \phi_1^1(t)+O(\epsilon^2)\\
\phi_2^{\epsilon}(t)&=&\phi_2^0+\epsilon \phi_2^1(t)+O(\epsilon^2).\\
\end{eqnarray*}
Using the above perturbation expansion in $\epsilon$, the time period $T$ Poincare map can be written as
\begin{eqnarray*}
&P_{\epsilon}:(I^\epsilon(0), \phi_1^\epsilon(0), \phi_2^\epsilon(0))\to (I^\epsilon(T), \phi_1^{\epsilon}(T),\phi_2^\epsilon(T))\\
&(I^0,\phi_1^0,\phi_2^0)\to (I^0+\epsilon I^1(T),\phi_1^0+\epsilon \omega_1(I^0)T+\epsilon \phi_1^1(T),\phi_2^0+ \omega_2(I^0)T+\epsilon \phi_2^1(T))+O(\epsilon^2).
\end{eqnarray*} %
 From this Poincare map, we are interested in the perturbations terms of order $\epsilon$ entering in $I$ and $\phi_1$ directions (i.e., $I^1(T)$ and $\phi_1^1(T)$) and verifying that their average with respect to $\phi_2^0$ is zero thereby satisfying Assumption \ref{assumption_1} of the main theorem. The perturbation terms of $O(\epsilon^2)$ and their zero average with respect to $\phi_2^0$ is not necessary because the averaging Lemma \ref{averaging_lemma}, where the Assumption \ref{assumption_1} of the main theorem is used, only reduces the size of perturbations from order $\epsilon$ to $\epsilon^2$.
We have following expressions for $I^1(T)$ and $\phi_1^1(T)$
\begin{eqnarray*}
&I^1(T)=\int_0^T \sin (\phi^0_2+\omega_2(I^0)\tau)\cos \varphi \sin \omega \tau G_{I}(I^0,\phi_1^0)d\tau\nonumber\\&-\int_0^T \cos (\phi^0_2+\omega_2(I^0)\tau)\sin \varphi \sin \omega \tau G_{I}(I^0, \phi_1^0)dt=:f_I(I^0,\phi_1^0,\phi_2^0)\nonumber\\
&\phi_1^1(T)=\int_0^T \sin (\phi^0_2+\omega_2(I^0)\tau)\cos \varphi \sin \omega \tau G_{\phi_1}(I^0,\phi_1^0)d\tau\nonumber\\&-\int_0^T \cos (\phi^0_2+\omega_2(I^0)\tau)\sin \varphi \sin \omega \tau G_{\phi_1}(I^0, \phi_1^0)d\tau=:f_{\phi_1}(I^0,\phi_1^0,\phi_2^0).
\end{eqnarray*}
Using the trigonometric identities for $\sin(a+b)$ and $\cos(a+b)$, it follows that
\[\int_0^{2\pi}f_I(I^0,\phi_1^0,\phi_2^0) d\phi_2^0=\int_0^{2\pi}f_{\phi_2}(I^0,\phi_1^0,\phi_2^0) d\phi_2^0=0.\]
This verifies that the Poincare map of system (\ref{sim_example_pert}) satisfies the Assumption \ref{assumption_1} of the main Theorem.

 We pursue a visualization technique based on ergodic partition to visualize the dynamics of this three-dimensional map.
 The basic idea behind the constructing of the ergodic partition is
to identify the set of points in the phase space which have same time averages for a set of basis functions \cite{Mezic:1994,mezic_chaos,marko_cdc,Levnajic_chaos}. We pursue the implementation of this idea as presented in \cite{marko_cdc}. Ideally these time averages are computed for a basis set  functions  defined on the phase space.  We provide a computational implementation using only finitely many functions. Fig. \ref{ergodic_partition}, shows the two dimensional slice of the ergodic partition in the three dimensional $(r,z,\theta)$ space. The two dimensional slice is taken at $\theta=0$  plane. The initial conditions for the time averages are chosen from the set $I=[0.2,0.3]\times [-0.1,0.1]\times \{0\}$. The number of initial conditions for the simulation are chosen to be equal to $7000$ and the total number of functions used for time averages equal $8^3=512$ . The averaging functions were selected as the truncated set of complex harmonics functions on the rectangle $D=[0,0.5]\times [-1,1]\times [0,2\pi]$ and are of the form
\[f_{\bar k}(x)=(2\pi)^{-\frac{3}{2}}e^{i2\pi\left<x,\bar k\right>},\]
where $\bar k\in [0,7]^3$ so that in each spatial direction up to $8$ harmonics are considered, and $x=T(R,z,\theta)$, with $T: D\rightarrow [0,1]^3$ consists of translation and rescaling of domain $D$.
For more details on the computation of ergodic partition refer to \cite{ mezic_chaos,marko_cdc,Levnajic_chaos}. In Fig. \ref{ergodic_partition}, we show the results of the computation for values of perturbation $\epsilon=0.05$ and $\epsilon=0.01$.  Given the finite color resolution and the computation of time average with finitely many functions we can only resolve the ergodic partition to finite approximation. However even with the finite resolution one can identify the signature of the surviving KAM tori as
smooth banded structure of the invariant sets shown in figure \ref{ergodic_partition}.


\begin{figure}[h] \centering \subfigure[]
{\includegraphics[width=2.2in]{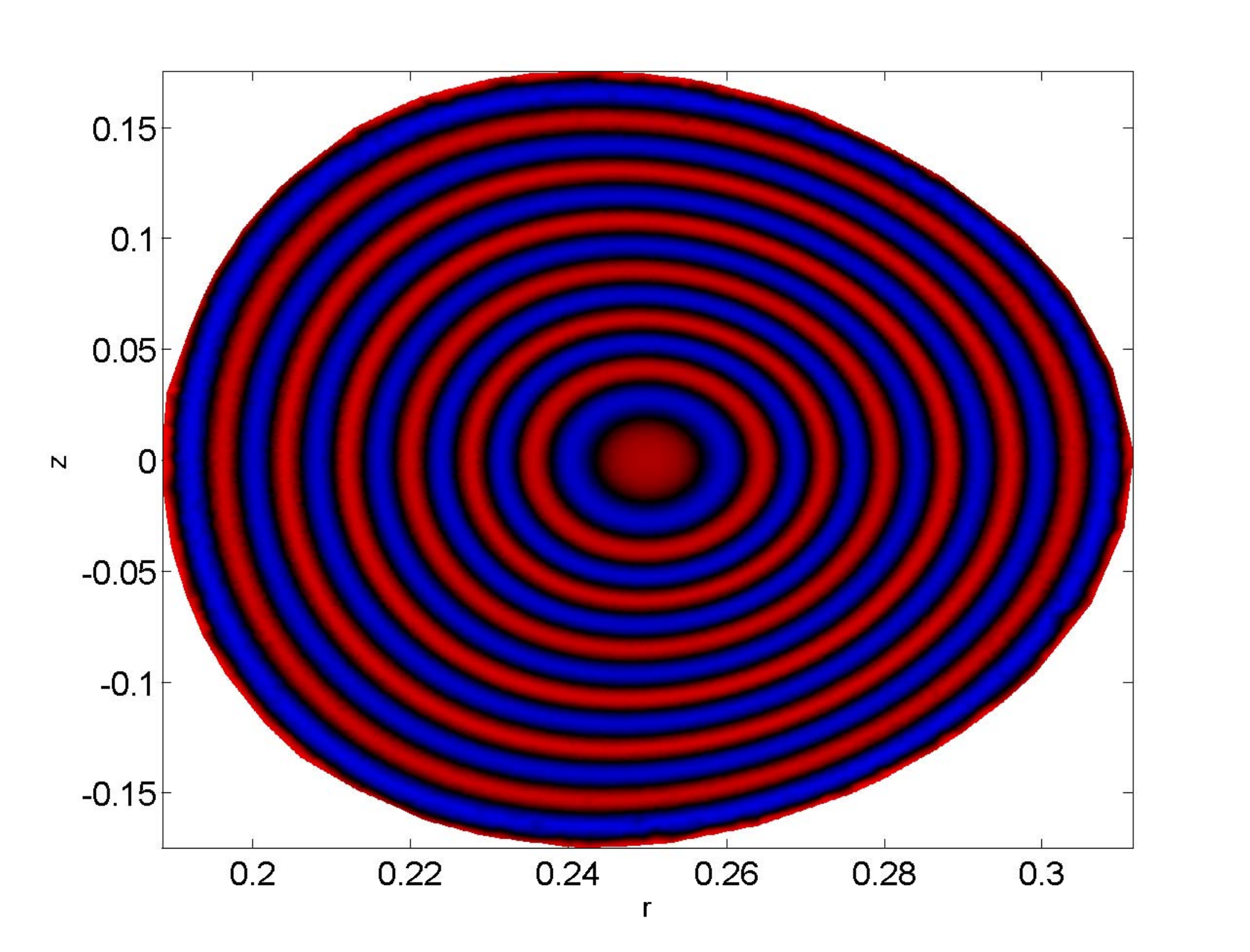}}
\hspace{.1in}
\subfigure[]
{\includegraphics[width=2.2in]{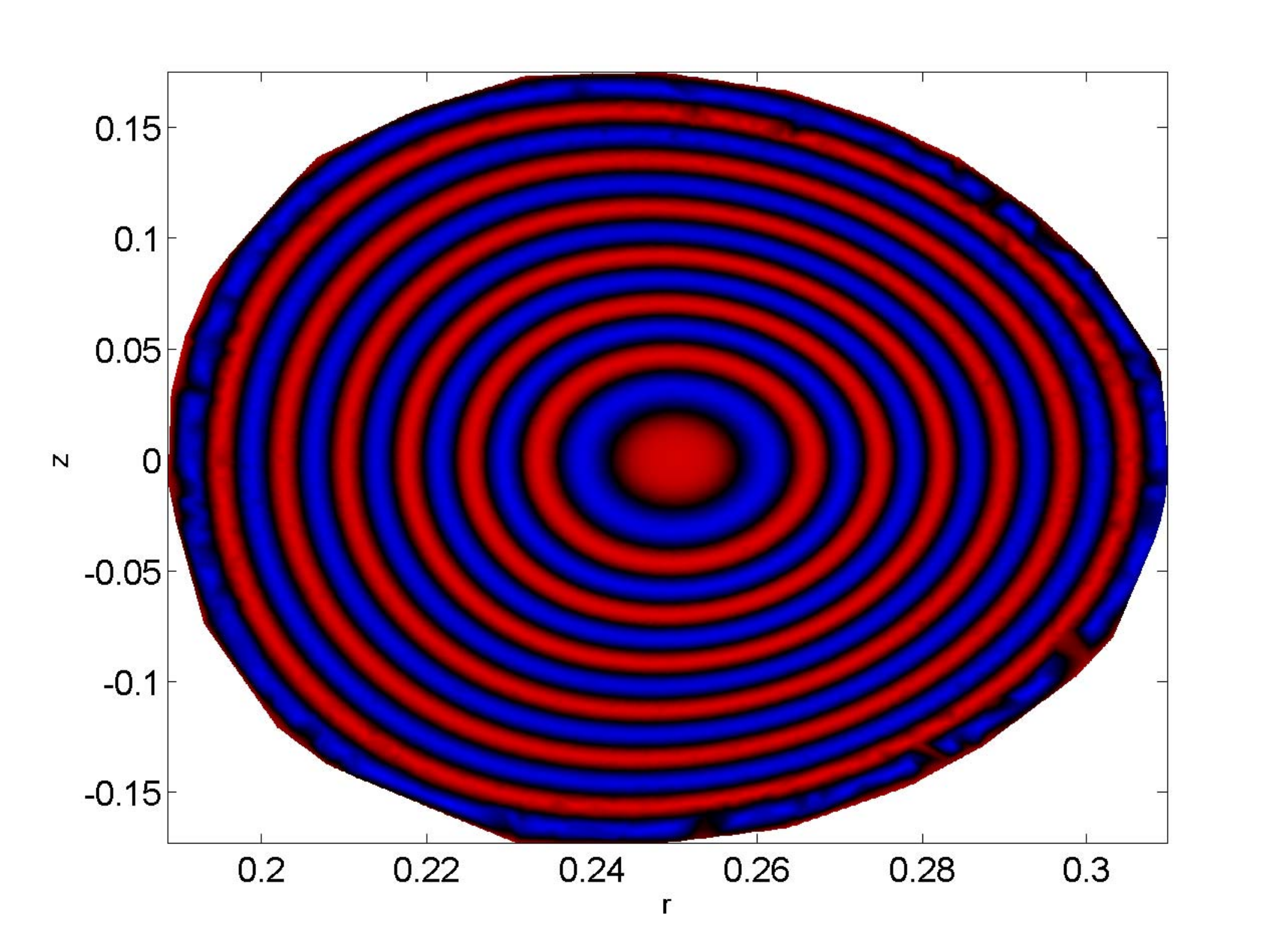}}
\caption{
\label{ergodic_partition} Visualization of ergodic partition on $\theta=0$ plane for the Poincare map of system (\ref{hills}); (a) $\epsilon= 0.05$; (b) $\epsilon=0.01$}\label{figure_tori}
\end{figure}

\section{Conclusions}
\label{conc}

In conclusion, we have proved the persistence of two-dimensional
invariant tori in the perturbation of integrable
action-angle-angle maps with degenerate angle. The persistence proof requires a
combination of the proof techniques for non-degenerate
volume-preserving maps as pursued in \cite{chengsun1} and
Arnold's methods in proving the KAM theorem in the case of Hamiltonian systems
with degenerate angles  \cite{arnoldsmallden, arnoldsmallden_2}. A specific peculiarity of our proof is the need for an intermediate sequence of coordinate transformations that reduces the size of the perturbation in the action variable by an order and allows us to proceed with a Moser-type technique pursued in  \cite{chengsun1}. Elegant and shorter proof technique for KAM-type results has recently been pursued by Broer, Huitema and Sevryuk \cite{broerbook}. It would be interesting to see whether their ``parametric" KAM-type technique could be used to prove a version of our theorem in a simpler way.
 In addition, we have used the main result of this paper to prove persistence of invariant tori in a perturbation  of a volume-preserving Euler fluid flow, swirling Hill's vortex, under the assumption of large swirl. Note that our proof above can be easily extended to the full class of perturbation similar to the single-mode perturbation in $\theta$ that we have used, as any such perturbation can be expanded in Fourier series. In other words, any sufficiently small, volume-preserving perturbation that has axial (z) symmetry will have a set of tori preserved.


\section{Acknowledgement}
The authors would like to acknowledge the help of Marko Budi\v{s}i\'{c}, from the University of California, Santa Barbara for providing the code and generating the plots for the ergodic partition of the three dimensional map in section \ref{Hill}. We thank an anonymous referee of the previous version for pointing out a problem with the proof. This research has been supported by ONR MURI grant.

%
%


\begin{thebibliography}{99}


\bibitem{kol}  A. N. Kolmogorov,  1954, On the conservation of conditionally periodic motions for a small change in
{H}amiltonian's function, {\it Doklady Akad. Nauk. SSSR.} {\bf 98}
p 527-530.


\bibitem{arnoldsmallden}   V. I. Arnold,  1963, Proof of a theorem of A. N. Kolmogorov on the preservation of conditionally periodic motions under a small perturbation of the Hamiltonian, {\it Russian Mathematical Surveys} {\bf 18}, no. 5,  p 9-36.



\bibitem{arnoldsmallden_2} V. I. Arnold, 1963, Small denominator and problem of stability of motion in classical and celestial mechanics, {\it Russian Mathematical Surveys } {\bf 18}, no.6, p 85-191.



\bibitem{moserkam} J. Moser,  1962, On invariant curves of area-preserving mappings of an annulus {\it Nachr. Akad. Wiss.
Gottingen Math Phys.} {\bf K1.II}, p 1-20.


\bibitem{SiegelandMoser:1971} C.  Siegel  and J. Moser, {\it Lectures on Celestial Mechanics} Berlin Heidelberg: Springer-
Verlag, 1971.





\bibitem{mezicwigg} I. Mezi\'{c}  and S. Wiggins, 1994, On the integrability and perturbation of three-dimensional fluid flow with
symmetry {\it J. Nonlinear Sci.} {\bf 4}, p 157-194.



\bibitem{RLlavebook}  R. de la Llave, {\it Recent progress in classical mechanics. Mathematical Physics, X (Leipzig 1991)} Berlin Heidelberg: Springer-Verlag 1992.



\bibitem{broer1} H.W. Broer, G. B. Huitema and F. Takens,  1990, Unfolding of quasi-periodic tori {\it Mem. Amer. Math. Soc.}
{\bf 421}, p 1-82.


\bibitem{broerbook} H. W. Broer, G. B. Huitema  and M. B. Sevryuk,  1996, Quasi-periodic motion in the families of dynamical systems: order amidst chaos,
{\it Lectures notes in Mathematics {\bf 1645}} Berlin Heidelberg:
Springer-Verlag.

\bibitem{Delshams} A. Delshams and  R. de la Llave, 1991, Existence of quasi-periodic orbits and absence of transport for
volume-preserving transformations and flow {\it Preprint}.



\bibitem{chengsun3} C. Q. Cheng  and Y. Sun,  1994, Existence of KAM tori in degenerate Hamiltonian systems {\it Journal of Differential Equations} {\bf
114}, p 288-335.







%



\bibitem{Xia} Z. H. Xia,  1992, Existence of invariant tori in volume-preserving diffeomorphisms {\it Ergodic Theory and
Dynamical Systems} {\bf 12}, p 621-631.




%

\bibitem{Zhu_Wen} Zhu Wen-Zhuang and Huang Qing-Dao and Liu Bai-Feng, 2004, The persistence of invariant tori in nearly small twist mappings with intersection
property {\it Northeast Math. J.} {\bf 20}, no.2, p 175-190.


\bibitem{Yong} Y. Li  and Y. Yi,  2002, Persistence of invariant tori in generalized Hamiltonian systems {\it Ergodic theory and
Dynamical Systems} {\bf 22}, p 1233-1261.


\bibitem{chengsun1} C. Q. Cheng  and Y. Sun, 1990, Existence of invariant tori in three-dimensional measure preserving mappings {\it Celestial
Mechanics} {\bf 47}, p 275-292.



\bibitem{Xiabook} Xia Zhihong, 1995, {\it Existence of invariant tori for certain non-symplectic diffeomorphism}, Hamiltonian
Dynamical Systems: History, Theory and Application New York:
Springer.


\bibitem{broer_LNM} H. W. Broer, 1981, Formal normal form theorems for  fields and some consequences of bifurcations in the volume preserving case {\it Lecture Notes in Mathematics 898} Springer-Verlag, p 54-74.



\bibitem{haller_mezic} G. Haller and I. Mezi\'{c}, 1998, Reduction of three-dimensional volume preserving flow with symmetry {\it Nonlinearity} {\bf 11}, p 319-339.


\bibitem{piro} O. Piro and M. Feingold, 1998, Diffusion in three-dimensional Liouvillian maps {\it Phys. Rev. Lett.} {\bf 61}, p
1799-1802.


\bibitem{mezicaaa} I. Mezi\'{c},  2001, Break-up of invariant surfaces in action-angle-angle maps and flow {\it Physica} D {\bf
154}, p 51-67.



\bibitem{whitney}  H. Whitney,  1934, Analytic extensions of differentiable functions defined in closed sets
 {\it Trans. Amer. Math. Soc.} {\bf 36}, p 63-89.





\bibitem{stein} E. M. Stein, 1970, {\it Singular integrals and differentiability properties of functions} Princeston New
Jersey: Princeton University Press.


\bibitem{Solomonandmezic:2003} T. H. Solomon and I. Mezi\'{c},   2003,  Uniform resonant chaotic mixing in fluid flows
 {\it Nature} {\bf 425}, no. 6956, p 376-380.



\bibitem{Mezic:1994} I. Mezi\'{c}, 1994, On  Geometrical and Statistical Properties of Dynamical Systems: Theory and Applications, PhD Thesis, {\it California Institute of Technology}.



\bibitem{mezic_chaos} I. Mezi\'{c} and S. Wiggins, 1999, A method for visualization of invariants sets of dynamical systems based on ergodic partition, {\it Chaos}  {\bf Vol. 9}, no. 1, p 213-218.





\bibitem{marko_cdc} M. Budi\v{s}i\'{c} and I. Mezi\'{c}, 2012, Geometry of the ergodic quotient reveals coherent structures in flows
{\it Accepted for publication in Physica D}.


\bibitem{Levnajic_chaos} Z. Levanaji\'{c} and I. Mezi\'{c}, 2010, Ergodic theory and visualization { I}: Mesochronic plots for visualization of
ergodic partition and invariant sets, {\it Chaos} {\bf Vol. 20}, no. 3. p (033114) 1-19.



%
\end{thebibliography}



%






\end{document}